\documentclass[12pt]{amsart}
\usepackage{amscd,amssymb}
\usepackage[arrow,matrix]{xy}
\usepackage[plainpages,backref,urlcolor=blue]{hyperref}

\theoremstyle{plain}
\newtheorem{thm}[subsection]{Theorem}

\newtheorem{prop}[subsection]{Proposition}

\newtheorem{conj}[subsection]{Conjecture}

\theoremstyle{definition}
\newtheorem{rk}[subsection]{Remark}

\newtheorem{ex}[subsection]{Example}

\numberwithin{equation}{section}
\newcommand{\R}{\mathbb{R}}
\newcommand{\C}{\mathbb{C}}

\newcommand{\PP}{\mathbb{P}}

\newcommand{\N}{\mathbb{N}}

\begin{document} 

\title [Codimension two complete intersections and Hilbert-Poincar\'e series]
{ Codimension two complete intersections and Hilbert-Poincar\'e series } 

\author[Gabriel Sticlaru]{Gabriel Sticlaru}
\email{gabrielsticlaru@yahoo.com } 



\begin{abstract} We investigate the relation between codimension two smooth complete intersections in a projective space and some naturally associated graded algebras. We give some examples of log-concave polynomials  and we propose two conjectures for these algebras.
\end{abstract}

\maketitle
 
\section{Introduction}

Let $S=\C[x_0,...,x_n]$ be the graded ring of polynomials in $x_0,,...,x_n$ with complex coefficients and denote by $S_r$ the vector space of homogeneous polynomials in $S$ of degree $r$. 
For any polynomial $f \in S_r$ we define the {\it Jacobian ideal} $J_f \subset S$ as the ideal spanned by the partial derivatives $f_0,...,f_n$ of $f$ with respect to $x_0,...,x_n$. 

The Hilbert-Poincar\'{e} series of a graded $S$-module $M$ of finite type is defined by 
\begin{equation} 
\label{eq2}
HP(M)(t)= \sum_{k\geq 0} \dim M_k\cdot t^k 
\end{equation} 
and it is known, to be a rational function of the form 
\begin{equation} 
\label{eq3}
HP(M)(t)=\frac{P(M)(t)}{(1-t)^{n+1}}=\frac{Q(M)(t)}{(1-t)^{d}}.
\end{equation}

For any polynomial $f \in S_r$ we define the corresponding graded {\it Milnor} (or {\it Jacobian}) {\it algebra} by
\begin{equation} 
\label{eq1}
M(f)=S/J_f.
\end{equation}
The hypersurface $V(f):f=0$ is smooth of degree $d$ if and only if $\dim M(f) <\infty$, and  then it is known that
\begin{equation} 
\label{HPsmooth}
HP(M(f))(t)=\frac{(1-t^{d-1})^{n+1}}{(1-t)^{n+1}}.
\end{equation}

Our examples and  the second conjecture refer to log-concavity and unimodal property. Before listing these examples we recall some basic notions.
 
A sequence $a_0,a_1, ..., a_m$ of  real numbers is said to be {\it log-concave} 
(resp. {\it strictly log-concave}) 
if it verifies $a_k^2 \geq a_{k-1}a_{k+1}$ (resp.    $a_k^2 > a_{k-1}a_{k+1}$) 
for $k=1,2,...,m-1$. An infinite sequence $a_k$, $k \in \N$ is (strictly) 
log-concave if any truncation of it is (strictly) log-concave.

 Such sequences play an important role in Combinatorics and Algebraic Geometry, see for instance the recent paper
\cite{H}. Recall that a sequence $a_0,...,a_m$ of real numbers is said to be {\it unimodal} if there is an integer $i$ between $0$ and $m$ such that
$$a_0 \leq a_1\leq ...\leq a_{i-1} \leq a_i \geq a_{i+1} \geq ...\geq a_m.$$
A nonnegative log-concave sequence with no internal zeros (i.e. a sequence for which the indices of the nonzero elements form a set of consecutive integers), is known to be unimodal, see \cite{H} and the references there.

A polynomial $P(t) \in \R[t]$, 
$P(t) = \sum_{j=0}^{j=m}a_jt^j,$
 with coefficients $a_j$ for $0\leq j \leq m$, is said to be log-concave (resp. unimodal) if the sequence of its coefficients is log-concave (resp. unimodal). For other examples with log-concave polynomials, see \cite{St1}.
For Singular language and applications, see \cite{DGP} and \cite{GP}.

\section{Main results}

Assume from now on that $V(f)$ is smooth of degree $d$ and let $g \in S_e$ be another homogeneous polynomial of degree $e$. We say that $V(f,g)=V(f) \cap V(g)$ is a smooth complete intersection if 
any point $p\in V(f,g)$ is smooth on both $V(f)$ and $V(g)$ and the corresponding tangent spaces $T_pV(f)$ and $T_p(V(g)$ are distinct. Our first result is the following.

\begin{thm}
\label{thm1} Let $m_2(f,g)$ be the ideal generated by all the $2 \times 2$ minors in the Jacobian matrix of the mapping $(f,g):\C^{n+1} \to\C^2.$ Define $A(f,g)= S/((f)+m_2(f,g))$ and $B(f,g)=
S/((f)+(g)+m_2(f,g))$.

With this notation, the following conditions are equivalent.

\medskip

\noindent (i) $V(f,g)=V(f) \cap V(g)$ is a smooth complete intersection.

\medskip

\noindent (ii) $\dim A(f,g)<\infty.$

\medskip

\noindent (iii) $\dim B(f,g)<\infty.$

\end{thm}

If $f=x_0$, then $ A(f,g)= B(f,g)=S'/J_{g'}=M(g')$, where 
$S'=\C[x_1,...,x_n]$ and $g'(x_1,...,x_n)=g(0,x_1,...,x_n)$.
However, in general we have the following.

\begin{prop}
\label{prop1} With the above notation, the ideals $I(f,g)= (f)+m_2(f,g)$ and $J(f,g)=
(f)+(g)+m_2(f,g)$ have the same radical, but are distinct in general.
When the equivalent conditions of  Theorem \ref{thm1} hold, this common radical is the maximal ideal $(x_0,...,x_n)$.
\medskip

\end{prop}
\begin {proof} (Theorem  \ref{thm1} and Proposition \ref{prop1}).

The equivalence of (i) and (iii) in Theorem  \ref{thm1} is a classical fact, see for instance Proposition (6.39) in \cite{D}. To prove the equivalence of (ii) and (iii) in Theorem  \ref{thm1}, it is enough to prove the first part of the claim in  Proposition \ref{prop1}, namely that the ideals
$I(f,g)$ and $J(f,g)$ have the same radical.
To do this, it is enough by Hilbert's Nulstellensatz, see for instance Theorem (2.14) in \cite{D}  to show that the corresponding zero sets
$V(I(f,g))$ and $V(J(f,g))$ in $\PP^n$ coincide. Moreover, it is clear that $V(J(f,g)) \subset  V(I(f,g))$.

Conversely, note that a point $p$ in $V(I(f,g))$  satisfies $f(p)=0$ and the differentials $df(p)$ and $dg(p)$ are proportional. Since $df(p)\ne 0$ as $V(f)$ is supposed to be smooth, it follows that there is $\lambda \in \C$ such that $dg(p)= \lambda df(p)$. This implies via the Euler formula that $g(p)=0$, hence $p \in V(J(f,g))$.

The second part of  the claim  in Proposition \ref{prop1} is dealt with in the next section by means of examples computed using Singular.

\end {proof}

\begin{rk}
\label{rk1}
Note that the equivalent condition in Theorem \ref{thm1} do not imply $\dim S/m_2(f,g) <\infty$.
Take for instance $f=x_0$ and $g=x_0^2+x_1^2+ ...+x_n^2.$
\end{rk}

As a consequence of Conjecture \ref{conj1} we can prove the following formulas for the Hilbert-Poincar\'e series of the graded $S$-modules $A(f,g)$ and $B(f,g)$.
The result is given in the form $HP(M)(t)=\frac{P(M)(t)}{(1-t)^4}$ and hence it is enough to give the polynomial $P(M)$.

\begin{prop}
\label{prop2} 

 The Hilbert-Poincar\'{e} series of a graded $S$-module $A(f,g)$ and $B(f,g)$ depend only on the degrees $d$ and $e$, when the equivalent conditions of Theorem \ref{thm1} hold. More precisely, one has the following in $\PP^3$.

\medskip

\noindent  (i) $P(A(f,g))(t)=1-[t^d+6t^{d+e-2}]+[4t^{d+2e-3}+4t^{2d+e-3}+6t^{2d+e-2}]-[t^{d+3e-4}+t^{2d+2e-4}+4t^{2d+2e-3}
+t^{3d+e-4}+4t^{3d+e-3}]+[t^{2d+3e-4}+t^{3d+2e-4}+t^{4d+e-4}]$.

\medskip

\noindent (ii) $P(B(f,g))(t)=1-[t^{e}+t^{d}+6t^{d+e-2}]+[4t^{d+e-1}+4t^{d+2e-3}+4t^{2d+e-3}]-[t^{d+3e-4}+t^{2d+2e-4}+4t^{2d+2e-3}
+t^{3d+e-4}]+[t^{2d+3e-4}+t^{3d+2e-4}]$.

\end{prop}
\begin {proof}
We explain now briefly how to get Proposition \ref{prop2} from Conjecture \ref{conj1}.

To get the formulas, we start with the resolution and get
$$HP(M)(t)=HP(S)(t)-HP(R_1)(t)+HP(R_2)(t)-HP(R_3)(t)+HP(R_4)(t),$$ 
where $HP(S)(t)=\frac{1}{(1-t)^4}$.

Then we use the well-known formulas $HP(N_1\oplus N_2)(t)=HP(N_1)(t)+HP( N_2)(t)$, 

$HP(N^p(-q))(t)=pt^qHP(N)(t)$, $HP(S(-k))(t)=\frac{t^k}{(1-t)^4}$. 

\end{proof}

\begin{rk}
\label{rk2}
Note that if $f=x$ and $g$ of degree $e$, we obtain the Hilbert-Poincar\'{e} series for any smooth hypersurfaces of degree $e$, in $\PP^2$:   

$S(t)=\frac{(1-t^{e-1})^{3}}{(1-t)^{3}}=(1+t+t^2+ \cdots +t^{e-2})^3.$
\end{rk}

\section{Examples and conjectures} 

The examples and the conjectures concern codimension $2$ complete intersections in $\PP^3$ with coordinates $(x:y:z:w)$.

\begin{ex} The polynomial $g$ does not belong to the ideal  $I(f,g)$ in general, i.e. $I(f,g) \ne J(f,g)$.
\label{ex1}
Let $V(f): f=x^2+y^2+z^2+w^2=0$ Fermat smooth surface of degree two and nodal surfaces $2A_1$ type, 
$V(g): g=xzw+(z+w)y^2+x^3+x^2y+xy^2+y^3=0$. 

Indeed Radical(I(f,g))=Radical(J(f,g))=Ideal(x,y,z,w).
The ideal $I(f,g)$ has seven generators, 

$I(f,g)=(G_1,G_2,G_3,G_4,G_5,G_6,G_7)$, where:

$G_1=2x^3-2x^2y+2xy^2-2y^3+4xyz+4xyw-2yzw$, 

$G_2=2xy^2-6x^2z-4xyz-2y^2z+2x^2w-2z^2w$,

$G_3=2xy^2+2x^2z-6x^2w-4xyw-2y^2w-2zw^2$, 

$G_4=2y^3-2x^2z-4xyz-6y^2z-4yz^2+2xyw-4yzw$,

$G_5=2y^3+2xyz-2x^2w-4xyw-6y^2w-4yzw-4yw^2$,

$G_6=2y^2z+2xz^2-2y^2w-2xw^2$,

$G_7=x^2+y^2+z^2+w^2$.  

One can check the identity:
$g=(7/2)G_1+(1/2)G_2+(-1/2)G_3+2G_4+(-2)G_5+3G_6+(-6x + 8y + 8z - 8w)G_7+
(-8z^3 - 28xyw + 2y^2w + xzw + 7yzw + 9z^2w + 12xw^2 - 16yw^2 - 9zw^2 + 8w^3)$
where Normal Form (remainder modulo $I(f,g)$) of $g$ is not zero, so $g$ does not belong to $I(f,g)$.

\end{ex}

\begin{ex} The Hilbert-Poincar\'{e} series of a graded $S$-module $A(f,g)$ and $B(f,g)$ depend only on the degrees $d$ and $e$,  when the equivalent conditions of Theorem  \ref{thm1} hold.
\label{ex2}

The table below  contain singular cubic projective surface in $\PP^3$ displayed with the type of singularity.

\noindent
\begin{center}
 \begin{tabular} {|l|l|}
  \hline                      
  Type      &   polynomial  $ g$ of degree three  \\ 
  \hline
$A_2$ 			&	$ (x+y+z)(x+2y+3z)w+xyz$ \\
$2A_1$ 			&	$ xzw+(z+w)y^2+x^3+x^2y+xy^2+y^3$\\
$A_1+A_2$ 	&	$ x^3+y^3+x^2y+xy^2+y^2z+xzw$ \\
$A_4$ 			&	$ y^2z+yx^2-z^3+xzw$ \\
$3A_1$ 			&	$ y^3+y^2(x+z+w)+4xzw $\\
$A_1+A_3$ 	&	$ wxz+(x+z)(y^2-x^2) $ \\ 
$A_5$ 			&	$	wxz+y^2z+x^3-z^3 $\\
$D_4$				&	$	w(x+y+z)^2+xyz $\\
$2A_1+A_2$	&	$	wxz+y^2(x+y+z) $\\
$A_1+A_4$		&	$	wxz+y^2z+yx^2 $\\
$D_5$				&	$	wx^2+xz^2+y^2z $\\
$4A_1$			&	$	w(xy+xz+yz)+xyz $\\
$A_1+2A_2$	&	$	wxz+xy^2+y^3 $\\
$2A_1+A_3$	&	$	wxz+(x+z)y^2 $ \\ 
$A_1+A_5$		&	$	wxz +y^2z+x^3 $\\ 
$E_6$				&	$	wx^2+xz^2+y^3 $\\
$3A_2$			&	$	wxz+y^3 			$\\
$\tilde{E_8}$   		& $ y^3+2z^3+4w^3  $ \\

\hline
\end{tabular} 
\end{center} 

\begin{enumerate}

\item 
Let $d=2, e=3$, $V(f): f=x^2+y^2+z^2+w^2=0$ and $V(g):g=0$ any singular cubic projective surface in $\PP^3$ from the table above.
In all cases, $g$ not belongs to $I(f,g)$, $I(f,g)$ and $J(f,g)$ have the same radical Ideal(x,y,z,w), but 
$A(f,g)$ and $B(f,g)$  are distinct, with different Hilbert-Poincar\'{e} finite series, log-concave polynomials.

$HP(A(f,g))(t)=1+4t+9t^2+10t^3+5t^4+t^5.$

$HP(B(f,g))(t)=1+4t+9t^2+ 9t^3+5t^4+t^5.$

\item 
Let $d=2, e=3$, $V(f): f=x^2+y^2+z^2+w^2=0$ and $V(g):g=0$ any singular cubic  projective surface in $\PP^3$ from the list below.

$A_3: $ $g=xzw+(x+z)(y^2-x^2-z^2)=0,$ 

$2A_2:$ $g=x^3+y^3+x^2y+xy^2+xzw=0.$\\
In these cases, $g$ not belongs to $I(f,g)$. $I(f,g)$ and $J(f,g)$ have the same radical other that Ideal(x,y,z,w), but $A(f,g)$ and $B(f,g)$  are distinct, with different Hilbert-Poincar\'{e} infinite series.

$HP(A(f,g))(t)=1+4t+9t^2+10t^3+5t^4+2(t^5+t^6+\cdots$,  

$HP(B(f,g))(t)=1+4t+9t^2+ 9t^3+5t^4+2(t^5+t^6+\cdots$. 

\item
Let $d=3, e=3$, $V(f): f=x^3+y^3+z^3+w^3=0$ and $V(g):g=0$ any singular cubic projective surface in $\PP^3$ from the table above.
In all cases, $g$ not belongs to $I(f,g)$, $I(f,g)$ and $J(f,g)$ have the same radical Ideal(x,y,z,w), but 
$A(f,g)$ and $B(f,g)$  are distinct, with different Hilbert-Poincar\'{e} finite series, log-concave polynomials.

$HP(A(f,g))(t)=1+4t+10t^2+19t^3+25t^4+22t^5+12t^6+3t^7.$

$HP(B(f,g))(t)=1+4t+10t^2+18t^3+21t^4+16t^5+8t^6+2t^7.$

\item
Let $d=1, e=3$, $V(f): f=x=0$ and $V(g):g=0$ any singular cubic projective surface in $\PP^3$ from the table above.
We obtain the Hilbert-Poincar\'{e} series for any smooth hypersurfaces of degree $3$ in $\PP^2$:  
$S(t)=\frac{(1-t^2)^{3}}{(1-t)^{3}}=1+3t+3t^2+t^3.$

\end{enumerate}

\end{ex}

\bigskip
These examples motivate our following conjectures.

The best way to understand the graded $S$-modules $A(f,g)$ and $B(f,g)$
is to construct their free resolutions. Based on the free resolution, we can compute the    
Hilbert-Poincar\'{e} series for each graded algebra.
We propose the following conjecture, concern codimension $2$ complete intersections in $\PP^3$
\begin{conj} 
\label{conj1}

Let $A(f,g)$ and $B(f,g)$ with $\dim A(f,g) and \dim B(f,g) <\infty$. 
Then the minimal graded free resolutions of these algebras are the following.
\begin{equation} 
\label{res}
0 \to R_4 {\rightarrow}  R_3 {\rightarrow} R_2 {\rightarrow} R_1 {\rightarrow} S \to M \to 0
\end{equation}

\medskip
\noindent (i) If M=$A(f,g)$, 

$R_1=S(-d)\oplus S^6[-(d+e-2)]$, 

$R_2=S^4[-(d+2e-3)]\oplus S^4[-(2d+e-3)]\oplus S^6[-(2d+e-2)],$

$R_3= S[-(d+3e-4)]\oplus S[-(2d+2e-4)] \oplus S^4[-(2d+2e-3)] \oplus S[-(3d+e-4)]$ 

$\oplus S^4[-(3d+e-3)], $

$R_4= S[-(2d+3e-4)]\oplus S[-(3d+2e-4)]\oplus S[-(4d+e-4)].$

\medskip

\noindent (ii) If M=$B(f,g)$, 

$R_1=S(-e)\oplus S(-d)\oplus S^6[-(d+e-2)],$

$R_2=S^4[-(d+e-1)]\oplus S^4[-(d+2e-3)]\oplus S^4[-(2d+e-3)],$

$R_3=S[-(d+3e-4)]\oplus S[-(2d+2e-4)]\oplus S^4[-(2d+2e-3)]\oplus S[-(3d+e-4)],$

$R_4=S[-(2d+3e-4)]\oplus S[-(3d+2e-4)].$
\medskip

\end {conj}

\begin{conj} 
\label{conj2}
When the equivalent conditions of Theorem  \ref{thm1} hold,
the Hilbert-Poincar\'{e} finite series
$HP(A(f,g))$ and $HP(B(f,g))$ are log-concave polynomials with no internal zeros, in particular  they are unimodal.

\end {conj}


\section{Conclusion}

The main results of this paper are Theorem  \ref{thm1} and Proposition \ref{prop1} in $\PP^n$ and Proposition \ref{prop2} in  $\PP^3$.
 
By analogy with Milnor algebra associated with one homogeneous polynomial in $\PP^n$, we define two graded algebras $A(f, g)$ and $B(f, g)$ associated with homogeneous polynomials $f$ and $g$, where projective hypersurface $V(f):f=0$ is smooth. We replace the {\it Jacobian ideal} with two ideals $I(f,g)$ and $J(f,g)$, generated by $f$, $g$ and $2 \times 2$ minors in the Jacobian matrix $(f,g)$. For these graded algebras we consider Hilbert-Poincar\'{e} series that encapsulates information about the dimensions of homogeneous components. When $f=x_0$, the  Hilbert-Poincar\'{e} series corresponds to the  smooth hypersurface in $\PP^{n-1}$, as in the Milnor algebra case. 

The main conclusion of the given examples \ref{ex1} and \ref{ex2} are Conjectures \ref{conj1} and \ref{conj2} in $\PP^3$.

When the equivalent conditions of Theorem \ref{thm1} hold, the Hilbert-Poincar\'{e} series of a graded module $A(f,g)$ and $B(f,g)$ depend only on the degrees of $f$ and $g$,
as in Milnor algebra smooth case, and are log-concave polynomials.

Finally, we consider this construction as a new method to generate many, infinite families of log-concave polynomials.

\bigskip

\noindent \textbf{Acknowledgment}\\[2mm] 

The author is grateful to A. Dimca for suggesting this problem.


\end{document}